\documentclass[11pt,reqno]{amsart}
\usepackage[utf8]{inputenc}
 
\usepackage{amssymb, amsmath, amsthm}

\usepackage[colorlinks]{hyperref}
\usepackage[nameinlink]{cleveref}
\hypersetup{
  linkcolor=[rgb]{0.3,0.3,0.6},
  citecolor=[rgb]{0.2, 0.6, 0.2},
  urlcolor=[rgb]{0.6, 0.2, 0.2}
}

\usepackage{verbatim}
\usepackage{amscd}    
\usepackage[all]{xy}  
\usepackage{youngtab} 
\usepackage{young}    
\usepackage{ytableau}
\usepackage{tikz-cd, tikz}
\usepackage{mathrsfs}
\usepackage{cases}
\usepackage{array}
\usepackage{cellspace}
\usepackage{calligra,mathrsfs}
\usepackage{bm}
\usepackage{graphicx}
\usepackage{rank-2-roots}
\usepackage{float}

\usepackage[style=alphabetic]{biblatex}
\usepackage{enumitem}
\setenumerate{label=\textup{(\roman*)}}

\usepackage[top=3.5cm,bottom=3.5cm,left=3cm,right=3cm]{geometry}

\DeclareMathOperator{\ShHom}{\mathscr{H}\text{\kern -3pt {\calligra\large om}}\,}

\def\initial{{\rm in}}
\def\Ker{{\mathrm{Ker}}}

\newtheorem{theorem}{Theorem}[section]
\newtheorem*{theorem*}{Theorem}
\newtheorem*{problem*}{Problem}
\newtheorem{lemma}[theorem]{Lemma}
\newtheorem{conjecture}[theorem]{Conjecture}

\newtheorem*{corollary*}{Corollary}

\newtheorem*{main-thm*}{Main Theorem}

\theoremstyle{definition}

\newtheorem*{definition*}{Definition}

\newtheorem{example}[theorem]{Example}

\theoremstyle{remark}
\newtheorem{remark}[theorem]{Remark}
\newtheorem*{remark*}{Remark}

\numberwithin{equation}{section}

\tikzset{
  treenode/.style = {align=center, inner sep=0pt, text centered,solid,thin,
    font=\sffamily},
  arn_n/.style = {treenode, circle, white, font=\sffamily\bfseries, draw=black,
    fill=black, text width=.5em},
  arn_nl/.style = {treenode, circle, white, font=\sffamily\bfseries, draw=black,
    fill=black, text width=1.5em},  
  arn_r/.style = {treenode, circle, red, draw=red, 
    text width=.5em, very thick},
  arn_v/.style = {treenode, circle, black, font=\sffamily\bfseries, draw=black, text width=1.2em},
  arn_x/.style = {treenode, rectangle, draw=black,
    minimum width=.5em, minimum height=0.5em},
  dott/.style={edge from parent/.style={dotted, very thick,circle,draw}},
  emph/.style={edge from parent/.style={dashed, very thick,circle,draw}},
  norm/.style={edge from parent/.style={solid,thin,circle,draw}}
}
\makeatletter
\def\labelbox#1{%
  \hbox{%
    \setbox\z@=\hbox{$\m@th\labelstyle{\,#1\,}$}%
    \setbox\tw@=\hbox{$\m@th\labelstyle\,$}%
    \dimen@=\ht\z@ \advance\dimen@ by \wd\tw@ \ht\z@=\dimen@
    \dimen@=\dp\z@ \advance\dimen@ by \wd\tw@ \dp\z@=\dimen@
    \box\z@
  }%
}
\makeatother
\begin{document}

\title[Join-meet binomial algebras of distributive lattices]{Join-meet binomial algebras of distributive lattices}

\author{Barbara Betti}
\address{(Barbara Betti) Otto-von-Guericke-University Magdeburg, Institut für Algebra und Geometrie, Universitätsplatz 2, Magdeburg, Germany.}
\email{barbara.betti@ovgu.de}

\author{Takayuki Hibi}
\address{(Takayuki Hibi) Department of Pure and Applied Mathematics, Graduate School
of Information Science and Technology, Osaka University, Suita, Osaka 565–0871,
Japan.}
\email{hibi@math.sci.osaka-u.ac.jp}

\date{}
\thanks{The research for this paper was initiated while the second author stayed at the Max Planck Institute for Mathematics in the Sciences, Leipzig, August 15 – September 8, 2025.}


\begin{abstract} We investigate the defining ideal of the algebra over a field generated by the join-meet binomials coming from a finite distributive lattice.  In the frame of algebras with straightening laws, the problem when the defining ideal is generated by quadrics is studied.  
\end{abstract}

\maketitle

\section*{Introduction}
\label{sec: intro}
Let $L$ be a finite lattice on $n$ elements $x_1,\dots, x_n$ and $S=K[x_1,\ldots,x_n]$ be the polynomial ring in $n$ variables over a field $K$. We are interested in the quadratic binomials 
\[ 
f_{ij}= x_ix_j - (x_i\vee x_j)(x_i\wedge x_j), \quad 1\leq i<j\leq n, 
\] 
called {\em join-meet binomials} of $L$.  Clearly, $f_{ij}$ is not zero if and only if $x_i$ and $x_j$ are incomparable in $L$.  The ideal $I_L=(f_{ij}:  1\leq i<j\leq n)$ and its quotient ring $S/I_L$ was introduced in \cite{hibi1987}.  In particular, by virtue of the classical results of Birkhoff \cite[Theorem 9.1.7]{HHgtm260} and Dedekind \cite[Chapter 3, Exercise 30]{EC1}, it is shown that $I_L$ is a prime ideal if and only if $L$ is a distributive lattice and that, when $L$ is a distributive lattice, $S/I_L$ is normal and Cohen--Macaulay.  Both the ideal $I_L$ and the quotient ring $S/I_L$ of a distributive lattice have been widely investigated thereafter. 

In this paper we adopt a different perspective and focus on the finitely generated $K$-algebra $$\mathcal{R}_K(L)=K[f_{ij}:   1\leq i<j\leq n, f_{ij} \neq 0] \subset S$$ of a distributive lattice $L$. We do not assume any restriction on the field $K$. We refer to $\mathcal{R}_K(L)$ as the \emph{join-meet algebra} of $L$. This class of algebras is of relevant interest as it contains important examples in commutative algebra and algebraic geometry such as the homogeneous coordinate rings of the Grassmannian ${\rm Gr}_K(2,m)$, for every $m\geq 4$. The defining ideal of $K[{\rm Gr}_K(2,m)]$ is generated by the well-known Plücker quadrics. In general, the defining ideal of $\mathcal R_K(L)$ contains quadrics which are independent of Pl\"ucker quadrics. For example, the join-meet algebra $\mathcal R_K(B_3)$  of the boolean lattice $B_3$ of rank $3$ (Figure \ref{fig: boolean lattice of rank 3}) has defining ideal generated by the following two quadrics:
\[ f_{3,6}f_{4,5} + f_{2,7}f_{4,6} -f_{3,5}f_{4,6} - f_{2,6}f_{4,7} \ =  \ f_{2,7}f_{3,5} - f_{2,7}f_{4,6} + f_{2,6}f_{4,7} - f_{2,3}f_{5,7}  \ = \ 0.\]

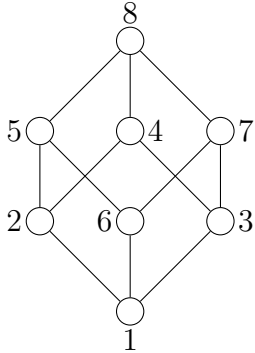
\begin{figure}
\centering
\begin{tikzpicture}[scale=1.2]
\node[draw,shape=circle] (1) at (1,-1) {};
\node[draw,shape=circle] (2) at (0,0) {};
\node[draw,shape=circle] (3) at (2,0) {};
\node[draw,shape=circle] (4) at (1,1) {};
\node[draw,shape=circle] (5) at (0,1) {};
\node[draw,shape=circle] (6) at (1,0) {};
\node[draw,shape=circle] (7) at (2,1) {};
\node[draw,shape=circle] (8) at (1,2) {};

\draw (1)node[below=1mm]{\large{1}}--(2)node[left=1mm]{\large{2}}--(5)node[left=1mm]{\large{5}}--(8)node[above=1mm]{\large{8}}--(7)node[right=1mm]{\large{7}}--(3)node[right=1mm]{\large{3}}--(1);
\draw (8)--(4)node[right=1mm]{\large{4}}--(2);
\draw (4)--(3);
\draw (1)--(6)node[left=1mm]{\large{6}}--(5);
\draw (6)--(7);
\end{tikzpicture}
\caption{Boolean lattice $B_3$ of rank $3$.}
\label{fig: boolean lattice of rank 3}
\end{figure}

A relevant result on this class of algebras is given in \cite{higashitani2026khovanskiibases}, where the authors characterize distributive lattices $L$ for which the set of generators $\mathcal{B}_L=\{ f_{ij}: 1\leq i<j\leq n, f_{ij} \neq 0\}$ is a Khovanskii/Sagbi basis with respect to a compatible term order. 

A final goal of our project is to classify those distributive lattices $L$ for which the defining ideal of $\mathcal{R}_K(L)$ is generated by quadrics. We work within the framework of algebras with straightening laws (ASL) \cite{Eisenbud1980}. As a first step in this direction, we classify distributive lattices $L$ whose join-meet algebra $\mathcal{R}_K(L)$ is a
polynomial ring (Theorem \ref{polynomial_ring}) and show that the defining ideal of $\mathcal{R}_K(L)$ of a thin distributive lattice is generated by quadrics (Theorem~\ref{highlight}). 

In the Appendix, assuming the Pl\"ucker relation \cite[p.~73]{bruns2022determinants} is known, we give a short proof of the classical fact (Hodge \cite{Ho} and Doubilet--Rota--Stein \cite{DRS}) that the homogeneous coordinate ring $K[x_iy_j - x_jy_i : 1 \leq i < j \leq n]$ of the Grassmann variety ${\rm Gr}_K(d,n)$ parameterizing $d$-dimensional subspaces of $K^n$ is an ASL on a distributive lattice over $K$.

\section{Partially ordered sets, distributive lattices and Gr\"obner bases}
    A partially ordered set is called a {\em poset}.  A {\em chain} of a finite poset $P$ is a totally ordered subset of $P$.  The {\em length} of a chain $C$ is $|C|-1$, where $|C|$ is the cardinality of $C$.  The {\em rank} of $P$ is the biggest length of chains of $P$. Let ${\rm rank}(P)$ denote the rank of $P$.  A finite poset is called {\em pure} if all of its maximal chains have the same length.

    A {\em finite lattice} is a finite poset whose any two elements $a,b$ have a unique greatest lower bound $a \wedge b$ and a unique least upper bound $a \vee b$, called, respectively, the {\em meet} and the {\em join} of $a$ and $b$. A finite lattice is {\em distributive} if meet and join operations distribute over each other.

    A subset $I$ of a finite poset $P$ is called a {\em poset ideal} of $P$ if it is downward closed, that is, for every $a \in I$ and $b\in P$ satisfying $b \leq a$ in $P$, we have $b \in I$.  In particular, $\emptyset$ and $P$ are poset ideals of $P$.  Let ${\mathcal J}(P)$ denote the finite poset consisting of all poset ideals of $P$, ordered by inclusion.  Then ${\mathcal J}(P)$ is a distributive lattice of rank $|P|$.   

    Let $L$ be a finite distributive lattice and $0_L$ (resp. $1_L$) be its unique minimal (resp. maximal) element.  We say that $x \in L$ is an {\em apex} of $L$ if $x$ is comparable to all $a \in L$. In particular, both $0_L$ and $1_L$ are apexes of L.  A finite distributive lattice $L$ is called {\em simple} if there is no apex of $L$ except for $0_L$ and $1_L$.  An element $0_L \neq a \in L$ is called {\em join-irreducible} if $a = b \vee c$, then either $a=b$ or $a=c$.  Let $P_L$ be the subposet of $L$ consisting of all join-irreducible elements of $L$.  Birkhoff's fundamental structure theorem for finite distributive lattices \cite[Theorem 9.1.7]{HHgtm260} guarantees that $L$ and ${\mathcal J}(P_L)$ are isomorphic as finite lattices.  We say that $L$ is {\em planar} if among any three join-irreducible elements $a,b,c\in L$, two of them are comparable in $L$.  The {\em{rank}} of $a \in L$ is the maximal length of chains of the form $a_{i_0} < a_{i_1} < \cdots < a_{i_r} = a$.  Let ${\rm rank}_L(a)$ denote the rank of $a \in L$.  Let $d = {\rm rank}(L) = {\rm rank}_L(1_L)$ and write $\rho_L(i)$ for the number of elements $a \in L$ with ${\rm rank}_L(a) = i$, where $0 \leq i \leq d$.  In particular, $\rho_L(0) = \rho_L(d)=1$.  Let $\theta(L) = \max\{\rho_L(i) : 1 \leq i < d \}$.  A finite simple distributive lattice $L$ is called {\em thin} \cite[p.~3]{HerzogHibi2023} (or a {\em generalized snake poset} \cite{BBHSVVY2022}) if $\theta(L) = 2$, i.e., $\rho_L(i) = 2$ for $1 \leq i < d$. It follows that a thin distributive lattice is planar.  

    Let $n > 0$ be an integer and let $D_n$ be the set of all divisors of $n$, ordered by divisibility. We call $D_n$ the {\em divisor lattice} of $n$. Every divisor lattice is distributive. A divisor lattice $D_n$ is called {\em Boolean} if $n$ is squarefree. A divisor lattice $D_n$ is thin if and only if $n = pq^s$, where $p $ and $ q$ are distinct primes and $s \geq 1$.

    The following lemma on Gröbner bases will be used in the next section. We refer the reader to \cite[Chapter 2]{HHgtm260} for fundamental materials on Gr\"obner bases. 

\begin{lemma}
    \label{lemma: grobner bases}
    Let $f_1, \ldots, f_s$ be homogeneous polynomials of $S=K[x_1, \ldots, x_n]$ of the same degree $>0$ and let $<$ be a monomial order on $S$. Consider the subrings $K[f_1,\ldots,f_s]$ and $K[{\rm in}_<(f_1), \ldots, {\rm in}_<(f_s)]$ of $S$.  Let $A=K[y_1, \ldots, y_s]$ be the polynomial ring in $s$ variables over $K$.  We define the surjective ring homomorphisms $\varphi:A \to K[f_1, \ldots, f_s]$ and $\psi:A \to K[{\rm in}_<(f_1), \ldots, {\rm in}_<(f_s)]$ by setting $\varphi(y_i)=f_i$ and $\psi(y_i) = {\rm in}_<(f_i)$.  Let $I={\rm Ker}(\varphi)$ denote the defining ideal of $K[f_1,\ldots,f_s]$ and $J={\rm Ker}(\psi)$ the defining ideal of $K[{\rm in}_<(f_1), \ldots, {\rm in}_<(f_s)]$.  Let $<^\sharp$ be a monomial order on $A$ and $G'=\{g_1,\ldots,g_t\}$ a Gr\"obner basis of $J$ with respect to $<^\sharp$. Suppose that ${\rm in}_{<^\sharp}(g_j) \in {\rm in}_{<^\sharp}(I)$ for each $1 \leq j \leq t$ and choose $h_j \in I$ with ${\rm in}_{<^\sharp}(g_j) = {\rm in}_{<^\sharp}(h_j)$. Then $G :=\{h_1, \ldots, h_t\}$ is a Gr\"obner basis of $I$ with respect to $<^\sharp$.
\end{lemma}

\begin{proof}
    Let $\mathcal H$ denote the set of those monomials $u \in A$ with $u \not\in ({\rm in}_{<^\sharp}(h_1), \ldots, {\rm in}_{<^\sharp}(h_t))$.  It suffices to prove that $\{\varphi(u) : u \in \mathcal H\}$ is a linearly independent set. Since    
    \[
    ({\rm in}_{<^\sharp}(h_1), \ldots, {\rm in}_{<^\sharp}(h_t))=({\rm in}_{<^\sharp}(g_1), \ldots, {\rm in}_{<^\sharp}(g_t))={\rm in}_{<^\sharp}(J),
    \]
    it follows that $\{\psi(u) : u \in \mathcal H\}$ is linearly independent.  Let $u_1, \ldots, u_\delta$ be monomials, where $u_\xi \neq u_{\xi'}$ if $\xi \neq \xi'$, belonging to ${\mathcal H}$ with $\delta >1$.  Thus $\psi(u_\xi) \neq \psi(u_{\xi'})$ if $\xi \neq \xi'$.  Suppose that 
\begin{eqnarray*}
    \label{LINEARLY}
\sum_{\xi=1}^\delta c_\xi\varphi(u_\xi) = 
\sum_{\xi=1}^\delta c_\xi f_1^{\alpha^{(1)}_{\xi}}\cdots f_s^{\alpha^{(s)}_{\xi}} =
0, \quad 0 \neq c_\xi \in K.
\end{eqnarray*}
We observe that 
\[
{\rm in}_<(f_1^{\alpha^{(1)}_{\xi}}\cdots f_s^{\alpha^{(s)}_{\xi}}) =
({\rm in}_<(f_1))^{\alpha^{(1)}_{\xi}}\cdots ({\rm in}_<(f_s))^{\alpha^{(s)}_{\xi}}=\psi(u_\xi)
\]
and, since $\psi(u_\xi) \neq \psi(u_{\xi'})$ whenever $\xi \neq \xi'$, there exists a unique $1 \leq \xi_0 \leq \delta$ for which
\begin{eqnarray*}
    \label{LINEARLYagain}
0={\rm in}_{<}\left(\sum_{\xi=1}^\delta c_\xi\varphi(u_\xi)\right) = c_{\xi_0} ({\rm in}_<(f_1))^{\alpha^{(1)}_{\xi_0}}\cdots ({\rm in}_<(f_s))^{\alpha^{(s)}_{\xi_0}}
= c_{\xi_0} \psi(u_{\xi_0}).
\end{eqnarray*}
    This is a contradiction, hence $G$ is a Gröbner basis of $I$ with respect to $<^\sharp$.
\end{proof}

\begin{remark}
    In Lemma \ref{lemma: grobner bases}, if ${\rm in}_<(f_1), \ldots, {\rm in}_<(f_s)$ are algebraically independent, then $f_1, \ldots, f_s$ are algebraically independent.  
\end{remark}

\section{Algebras with straightening laws}
    Let $R = \bigoplus_{n=0}^{\infty} R_n$ be a noetherian graded algebra over $K$.  Let $P$ be a finite poset and suppose that an injection $\varphi: P \hookrightarrow \bigcup_{n=1}^{\infty} R_n$ for which the $K$-algebra $R$ is generated by $\varphi(P)$ over $K$ is given.  A {\em standard monomial} is a homogeneous element of $R$ of the form $\varphi(\gamma_1) \varphi(\gamma_2)\cdots \varphi(\gamma_n)$, where $\gamma_1 \leq \gamma_2 \leq \cdots \leq \gamma_n$ in $P$.  We call $R$ an {\em algebras with straightening laws} \cite{Eisenbud1980} on $P$ over $K$ if the following conditions are satisfied:
\begin{itemize}
\item[]
(ASL\,-1)
The set of standard monomials is a basis of $R$ over $K$;
\item[]
(ASL\,-2)
If $\alpha$ and $\beta$ in $P$ are incomparable and
\begin{eqnarray}
\label{ASL}
\, \, \, \, \, \, \, \, \, \, \, \, \, \, \, \, \, \, \, \, 
\varphi(\alpha)\varphi(\beta) 
= \sum_{i} r_i\,\varphi(\gamma_{i_1})\varphi(\gamma_{i_2}) \cdots , \, \, \, 0 \neq r_i \in K, \, \, \, \gamma_{i_1}\leq \gamma_{i_2} \leq \cdots 
\end{eqnarray}
    is the unique expression of $\varphi(\alpha)\varphi(\beta) \in R$ as linear combination of distinct standard monomials guaranteed by (ASL\,-1), then $\gamma_{i_1} \leq \alpha, \beta$ for every $i$. 
\end{itemize}
    The right-hand side of the relation in (ASL\,-2) is allowed to be the empty sum $(=0)$.  We abbreviate an algebra with straightening laws as ASL.  The relations in (ASL\,-2) are called the {\em straightening relations} for $R$.  

    Let $A=K[x_{\alpha} : \alpha \in P]$ denote the polynomial ring in $|P|$ variables over $K$ and define the surjection $\pi: A \to R$ by setting $\pi(x_\alpha) = \varphi(\alpha)$.  The defining ideal $I_R$ of $R = \bigoplus_{n=0}^{\infty} R_n$ is the kernel $\Ker(\pi)$ of $\pi$.  If $\alpha$ and $\beta$ in $P$ are incomparable, then we introduce the polynomial of $A$ 
\[
    f_{\alpha,\beta}:=x_\alpha x_\beta - \sum_{i} r_i\,x_{\gamma_{i_1}}x_{\gamma_{i_2}} \cdots
\]
    arising from (ASL-2) and belonging to $I_R$.  Let $G_R$ denote the set of those $f_{\alpha,\beta}$ for which $\alpha$ and $\beta$ are incomparable in $P$.  Let $<_{\rm rev}$ denote the reverse lexicographic order on $A$ induced by an ordering of the variables for which $x_{\alpha} <_{\rm rev} x_{\beta}$ if $\alpha < \beta$ in $P$.  It follows that $\initial_{<_{\rm rev} }(f_{\alpha,\beta}) = x_\alpha x_\beta$ and $G_R$ is a Gr\"obner basis of $I_R$ with respect to $<_{\rm rev}$.  Furthermore, $\dim R = {\rm rank}(P)+1$.

    In the definition of ASL, if we require only (ASL\,-2), then we call $R$ a {\em weakly} ASL on $P$ over $K$.  It would be of interest to find a criterion for a weakly ASL on $P$ over $K$ to be an ASL on $P$ over $K$.  One possible and standard way is to compute Hilbert functions.  On the other hand, by virtue of Lemma \ref{lemma: grobner bases}, we have the following result. 

\begin{lemma}
    \label{weakly}
    Let $f_1, \ldots, f_s$ be homogeneous polynomials of $S=K[x_1, \ldots, x_n]$ of the same degree $>0$ and let $<$ be a monomial order on $S$. Consider the subrings $K[f_1,\ldots,f_s]$ and $K[{\rm in}_<(f_1), \ldots, {\rm in}_<(f_s)]$ of $S$.  Let $P=\{p_1, \ldots, p_s\}$ be a finite poset and define the injection $\pi:P \to K[f_1, \ldots, f_s]$ and $\pi':P \to K[{\rm in}_<(f_1), \ldots, {\rm in}_<(f_s)]$ by setting $\pi(p_i)=f_i$ and $\pi'(p_i) = {\rm in}_<(f_i)$.  Suppose that $K[{\rm in}_<(f_1), \ldots, {\rm in}_<(f_s)]$ is an ASL on $P$ over $K$ and that $K[f_1,\ldots,f_s]$ is a weakly ASL on $P$ over $K$.  Then $K[f_1,\ldots,f_s]$ is an ASL on $P$ over $K$. 
\end{lemma}

\begin{proof}
    Let $G'=\{ f_{\alpha,\beta}' : \alpha,\beta \text{ uncomparable} \}$ be the Gröbner basis of $\Ker(\pi')$ defined via the straightening relations of $K[{\rm in}_<(f_1), \ldots, {\rm in}_<(f_s)]$. Then the polynomials $f_{\alpha,\beta}$, defined after the straightening relations of $K[f_1,\ldots,f_s]$, share the initial terms with the polynomials of $G'$, hence they form a Gröbner basis for $\Ker(\pi)$ by Lemma \ref{lemma: grobner bases}. Since the defining ideal of $K[f_1,\ldots,f_s]$ has a Gröbner basis consisting of straightening relations, the standard monomials are linearly independent and $K[f_1,\ldots,f_s]$ is ASL on $P$ over $K$.
\end{proof}

\begin{example} 
\label{hibi87} 
    Let $P_n=\{(i, j) : 1 \leq i,j \leq n\}$ denote the finite poset whose partial order is defined by setting $(i,j) \leq (i',j')$ if $i \leq i'$ and $j \leq j'$.  Let $S=K[x_1, \ldots, x_{n}, y_1, \ldots, y_n]$ denote the polynomial ring in $2n$ variables over $K$ and $R_n = K[x_iy_j : 1 \leq i,j \leq n] $.  After defining the injection $\pi:P_n \to R_n$ by setting $\pi((i,j))=x_{i}y_{j}$, it follows from the discussion done in \cite[pp.~98--99]{hibi1987} that $R_n$ is an ASL on $P_n$ over $K$.
\end{example}

    Let $L = \{0_L, x_1, \ldots, x_n, y_1, \ldots, y_n, 1_L\}$ be a thin distributive lattice with $n + 1 = {\rm rank}(L)$, where $0_L < x_1 < \cdots < x_n <1_L$ and $0_L < y_1 < \cdots < y_n <1_L$ (Figure \ref{fig: thin distributive lattice}).  Note that, once $x_1$ and $y_1$ are fixed, the assignment of $x_i$ and $y_i$ for $2 \leq i \leq n$ is automatically determined.  Let $S=K[x_1, \ldots,x_n, y_1,\ldots,y_n]$ denote the polynomial ring in $2n$ variables over $K$.  Recall that we associate $x_i$ and $y_j$ which are incomparable in $L$ with the quadratic binomial 
\[
    f_{ij} = x_iy_j - (x_i \wedge y_j)(x_i\vee y_j) 
\]
    of $S$. We fix a monomial order $<$ on $S$ for which ${\rm in}_<(f_{ij}) = x_iy_j$ for all $f_{ij}$.  We introduce the subalgebra ${\mathcal R}_K(L)$ of $S$ which is generated by all binomials $f_{ij}$.  Let $Q_L$ be the sublattice of $P_n$ of Example \ref{hibi87} consisting of the pairs $(i,j)$ with $f_{ij} \in {\mathcal R}_K(L)$ (Figure \ref{fig: sublattice Q_L}).

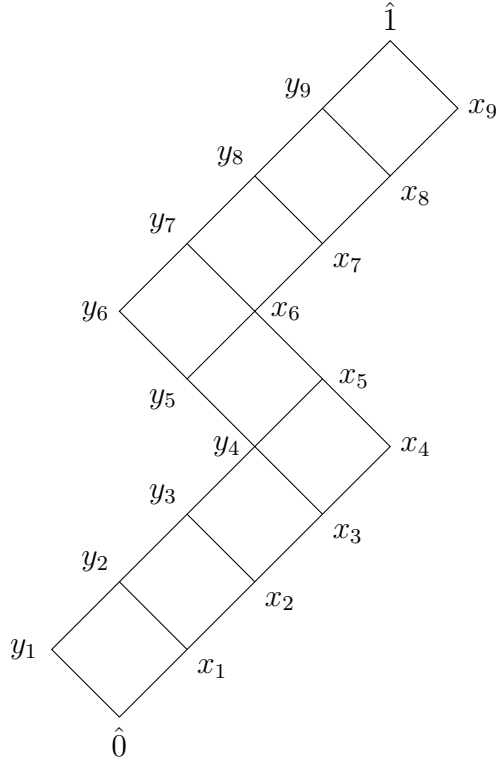
\begin{figure} 
    \centering
    \begin{tikzpicture}[scale=0.9]
    \begin{scope}[rotate around z=0]
    \coordinate (0) at (1,1) {};
    \coordinate (x1) at (2,2) {};
    \coordinate (x2) at (3,3) {};
    \coordinate (x3) at (4,4) {};
    \coordinate (x4) at (5,5) {};
    \coordinate (x5) at (4,6) {};
    \coordinate (x6) at (3,7) {};
    \coordinate (x7) at (4,8) {};
    \coordinate (x8) at (5,9) {};
    \coordinate (x9) at (6,10) {};
    \coordinate (1) at (5,11) {};
    \coordinate (y9) at (4,10) {};
    \coordinate (y8) at (3,9) {};
    \coordinate (y7) at (2,8) {};
    \coordinate (y6) at (1,7) {};
    \coordinate (y5) at (2,6) {};
    \coordinate (y4) at (3,5) {};
    \coordinate (y3) at (2,4) {};
    \coordinate (y2) at (1,3) {};
    \coordinate (y1) at (0,2) {};
    \draw(0)node[below]{\large{$\hat0$}}--
    (x1)node[below right]{\large{$x_1$}}--
    (x2)node[below right]{\large{$x_2$}}--
    (x3)node[below right]{\large{$x_3$}}--
    (x4)node[right]{\large{$x_4$}}--
    (x5)node[right=0.8mm]{\large{$x_5$}}--
    (x6)node[right=0.8mm]{\large{$x_6$}}--
    (x7)node[below right]{\large{$x_7$}}--
    (x8)node[below right]{\large{$x_8$}}--
    (x9)node[right]{\large{$x_9$}}--
    (1)node[above]{\large{$\hat1$}}--
    (y9)node[above left=0mm]{\large{$y_9$}}--
    (y8)node[above left=0mm]{\large{$y_8$}}--
    (y7)node[above left=0mm]{\large{$y_7$}}--
    (y6)node[left=0mm]{\large{$y_6$}}--
    (y5)node[below left=0mm]{\large{$y_5$}}--
    (y4)node[left=0.5mm]{\large{$y_4$}}--
    (y3)node[above left=0mm]{\large{$y_3$}}--
    (y2)node[above left=0mm]{\large{$y_2$}}--
    (y1)node[left=0.5mm]{\large{$y_1$}}--(0);
    \draw(x1)--(y2);
    \draw(x2)--(y3);
    \draw(x3)--(y4);
    \draw(x5)--(y4);
    \draw(x6)--(y5);
    \draw(x6)--(y7);
    \draw(x7)--(y8);
    \draw(x8)--(y9);
    \end{scope}
    \end{tikzpicture}
    \caption{A thin distributive lattice. } \label{fig: thin distributive lattice}
\end{figure}

\begin{figure} 
    \centering
    \begin{tikzpicture}[scale=0.9]
    \draw[step=1cm,gray,line width = 0.5](0,8) grid (8,0);
    \draw (0.1,-0.3) node{\large{$x_1$}};
    \draw (1,-0.3) node{\large{$x_2$}};
    \draw (2,-0.3) node{\large{$x_3$}};
    \draw (3,-0.3) node{\large{$x_4$}};
    \draw (4,-0.3) node{\large{$x_5$}};
    \draw (5,-0.3) node{\large{$x_6$}};
    \draw (6,-0.3) node{\large{$x_7$}};
    \draw (7,-0.3) node{\large{$x_8$}};
    \draw (8,-0.3) node{\large{$x_9$}};
    \draw (-0.3,0.1) node{\large{$y_1$}};
    \draw (-0.3,1) node{\large{$y_2$}};
    \draw (-0.3,2) node{\large{$y_3$}};
    \draw (-0.3,3) node{\large{$y_4$}};
    \draw (-0.3,4) node{\large{$y_5$}};
    \draw (-0.3,5) node{\large{$y_6$}};
    \draw (-0.3,6) node{\large{$y_7$}};
    \draw (-0.3,7) node{\large{$y_8$}};
    \draw (-0.3,8) node{\large{$y_9$}};
    \fill(0,0)circle(0.7mm);
    \fill(1,0)circle(0.7mm);
    \fill(2,0)circle(0.7mm);
    \fill(3,0)circle(0.7mm);
    \fill(1,1)circle(0.7mm);
    \fill(2,1)circle(0.7mm);
    \fill(3,1)circle(0.7mm);
    \fill(2,2)circle(0.7mm);
    \fill(3,2)circle(0.7mm);
    \fill(3,3)circle(0.8mm);
    \fill(3,4)circle(0.8mm);
    \fill(4,4)circle(0.8mm);
    \fill(3,5)circle(0.8mm);
    \fill(4,5)circle(0.8mm);
    \fill(5,5)circle(0.8mm);
    \fill(6,5)circle(0.8mm);
    \fill(7,5)circle(0.8mm);
    \fill(8,5)circle(0.8mm);
    \fill(6,6)circle(0.8mm);
    \fill(7,6)circle(0.8mm);
    \fill(8,6)circle(0.8mm);
    \fill(7,7)circle(0.8mm);
    \fill(8,7)circle(0.8mm);
    \fill(8,8)circle(0.8mm);
    \draw[line width=1.5](0,0)--(3,0)--(3,4)--(4,4)--(4,5)--(8,5)--(8,8);
    \draw[line width=1.5](1,0)--(1,1)--(3,1);
    \draw[line width=1.5](2,0)--(2,1);
    \draw[line width=1.5](2,1)--(2,2)--(3,2);
    \draw[line width=1.5](3,4)--(3,5)--(4,5);
    \draw[line width=1.5](6,5)--(6,6)--(8,6);
    \draw[line width=1.5](7,5)--(7,7)--(8,7);
    \end{tikzpicture}
    \caption{The sublattice $Q_L$ of the thin distributive lattice $L$ of Figure \ref{fig: thin distributive lattice}} \label{fig: sublattice Q_L}
\end{figure}
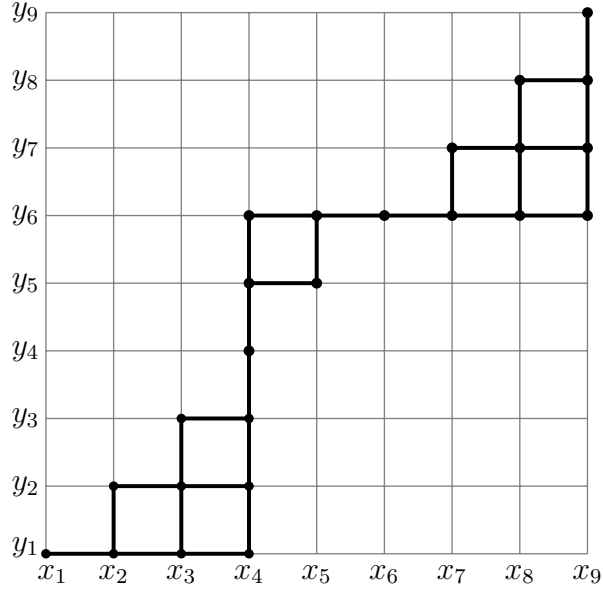

    Lemma \ref{structure} below is a simple consequence of Birkhoff's fundamental structure theorem for finite distributive lattices \cite[Theorem 9.1.7]{HHgtm260}. 

\begin{lemma}
    \label{structure}
    The partial order on $\{x_{i-1},y_{i-1},x_n,y_n\}$ is either $x_{i-1} < x_i, y_{i-1} < y_i, x_{i-1} < y_i$ or $x_{i-1} < x_i, y_{i-1} < y_i, y_{i-1} < x_i$.
\end{lemma}

\begin{lemma}
    \label{chain}
    The subposet $Q_L$ of a thin distributive lattice $L$ of rank $n+1$ is a chain of the poset $P_n$ of Example \ref{hibi87} if and only if the divisor lattice $D_{2\cdot 3^3}$ is not a sublattice of $L$.
\end{lemma}

\begin{proof}
    Suppose that $D_{2\cdot 3^3} = \{2^p\cdot 3^q : p\in \{0,1\}, q \in \{0,1,2,3\} \}$ is a sublattice of $L$.  Then $(2,3^2) < (2\cdot 3, 3^2) < (2\cdot 3, 3^3)$ and $(2,3^2) < (2, 3^3) < (2\cdot 3, 3^3)$.  Thus $Q_L$ cannot be a chain.

    Now, suppose that the divisor lattice $D_{2\cdot 3^3}$ is not a sublattice of $L$. Let, say, $x_{i-2} < x_{i-1}, x_{i-2} < y_{i-1}, y_{i-2} < y_{i-1}$.  Working with induction on $n$ enables us to assume that the sublattice $Q_{L'}$ with $L'=L \setminus \{x_i,y_i\}$ is a chain of $P_{n}$ whose maximal element is $(x_{i-1},y_{i-1})$.  Since $y_{i-1} < x_{i}, y_{i-1} < y_{i}, x_{i-1} < x_{i}$, it follows that $Q_L = Q_{L'} \cup \{(x_{i-1},y_i), (x_i,y_i)\}$ is a chain of $P_n$, as desired.
\end{proof}

\begin{example}
    \label{ASL_again}
    Let $L=D_{2 \cdot 3^{n-1}} = \{\hat{0}, x_1, \ldots, x_{n-1}, y_1, \ldots, y_{n-1}, \hat{1}\}$ with $2 = y_1, 3= x_1$ as above.  If $\hat{0} = x_0$ and $\hat{1} = y_n$.  Let $(\ell,i)$ and $(k,j)$ are incomparable in $Q_L$, say, $i < j \leq k < \ell$, then the classical Pl\"ucker relation yields
    \begin{eqnarray*}
    &&(y_ix_{\ell} - y_{\ell+1} x_{i-1})(y_jx_{k} - y_{k+1} x_{j-1})  \\
    & = &(y_ix_{k} - y_{k} x_{i-1})(y_jx_{\ell} - y_{\ell} x_{j-1})-(y_ix_{j -1} - y_{j} x_{i-1})(y_{k+1}x_{\ell} - y_{\ell+1} x_{k}),    
    \end{eqnarray*}
    In other words, we have a quadratic relation among the binomial generators of $\mathcal{R}_K(L)$:
    \[
    f_{\ell,i}f_{k,j} = f_{k,i}f_{\ell,j} - f_{j-1,i}f_{\ell,k+1}.
    \]
    Since $(k,i)<(\ell,i), (j-1,i)<(\ell,i), (k,i)<(k,j),(j-1,i)<(k,j)$, it follows that ${\mathcal R}_K(L)$ is a weakly ASL on $Q_L$ over $K$.
\end{example}

\begin{theorem}
    \label{finally}
    The join-meet algebra ${\mathcal R}_K(L)$ of a thin distributive lattice $L$ is a weakly ASL on $Q_n$ over $K$.    
\end{theorem}

\begin{proof}
    Assume $x_{1} < x_{2}, x_{1} < y_{2}, y_{1} < y_{2}$ and let $i_0$ be the smallest integer $i$ for which
    $x_{i-2} < x_{i-1}, x_{i-2} < y_{i-1}, y_{i-2} < y_{i-1}$
    and
    $y_{i-1} < x_{i}, y_{i-1} < y_{i}, x_{i-1} < x_{i}$.  Let $L_{(-)}$ denote the interval $[\hat{0}, x_{i_0}]$ of $L$ and $L_{(+)}$ the interval $[x_{i_{0}-2},\hat{1}]$ of $L$.  Then $(x_{i_0-1},y_{i_0-1})$ is the unique maximal element of  $Q_{L_{(-)}}$ as well as the unique minimal element of $Q_{L_{(+)}}$.  Furthermore, each $\alpha \in Q_{L_{(-)}}$ and each $\beta \in Q_{L_{(+)}}$ are comparable in $L$.  Thus $(x_{i_0-1},y_{i_0-1})$ is an apex of $Q_L$.  

    Now, Example \ref{ASL_again} says that ${\mathcal R}_K(L_{(-)})$ is a weakly ASL on $Q_{L_{(-)}}$ over $K$ and, in addition, working with induction on $n$, it follows that ${\mathcal R}_K(L_{(+)})$ is a weakly ASL on $Q_{L_{(+)}}$ over $K$.  Hence ${\mathcal R}_K(L)$ is a weakly ASL on $Q_L$ over $K$, as required. 
\end{proof}

\section{Relations of join-meet binomials}
    Let $L$ be a finite distributive lattice on $n$ elements $x_1, \ldots, x_n$ and $S=K[x_1, \ldots, x_n]$ be the polynomial ring in $n$ variables over a field $K$.  Recall that the join-meet binomials of $L$ are
\[ 
    f_{ij}= x_ix_j - (x_i\vee x_j)(x_i\wedge x_j), \quad 1\leq i<j\leq n.
\] 
    Let ${\mathcal R}_K(L)$ be the join-meet algebra 
\[ \mathcal{R}_k(L)=k[f_{ij}:   1\leq i<j\leq n, f_{ij} \neq 0] \subset S \]
    of $L$.  Let $A=K[y_{i,j} : 1\leq i<j\leq n, f_{ij} \neq 0]$ denote the polynomial ring over $K$ and define the surjective ring homomorphism $\pi : A \to {\mathcal R}_K(L)$ by setting $\pi(y_{ij})=f_{ij}$.  The defining ideal $I_{{\mathcal R}_K(L)}$ of ${\mathcal R}_K(L)$ is the kernel of $\pi$. 
    It is natural to ask for which distributive lattices $L$ the join-meet algebra ${\mathcal R}_K(L)$ is isomorphic to a polynomial ring, that is, the join-meet binomials are algebraically independent.  

\begin{example} \label{ex: D22}
    The divisor lattice $D_{2^2\cdot 3^2}$ (Figure \ref{fig: divisor lattice 2^2 3^2}) has $9$ binomial generators. These are the $2$-minors of a $3\times 3$ matrix.  It follows from \cite[Theorem 6.4.7 (c)]{bruns2022determinants} that the join-meet algebra $\mathcal R_K(D_{2,2})$ is a polynomial ring.

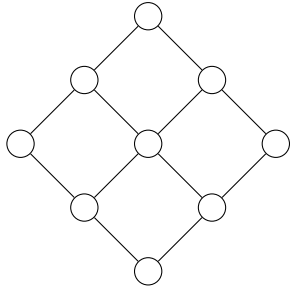
\begin{figure} 
    \centering
    \begin{tikzpicture}[scale=1.2]
    \begin{scope}[rotate around z=45]
    \node[draw,shape=circle] (1) at (0,0) {};
    \node[draw,shape=circle] (2) at (1,0) {};
    \node[draw,shape=circle] (3) at (2,0) {};
    \node[draw,shape=circle] (4) at (0,1) {};
    \node[draw,shape=circle] (5) at (1,1) {};
    \node[draw,shape=circle] (6) at (2,1) {};
    \node[draw,shape=circle] (7) at (0,2) {};
    \node[draw,shape=circle] (8) at (1,2) {};
    \node[draw,shape=circle] (9) at (2,2) {};
    \end{scope}
    \draw(1)--(2)--(3)--(6)--(9)--(8)--(7)--(4)--(1);
    \draw(2)--(5)--(8);
    \draw(4)--(5)--(6);
    \end{tikzpicture}
    \caption{Divisor lattice $D_{2^2\cdot 3^2}$.} \label{fig: divisor lattice 2^2 3^2}
\end{figure}
\end{example}

\begin{theorem} \label{polynomial_ring}   
    Let $L$ be a distributive lattice.  Then ${\mathcal R}_K(L)$ is a polynomial ring if and only if $L$ is planar and the divisor lattice $D_{2\cdot3^3}$ is not a sublattice of $L$.
\end{theorem}

\begin{proof}
    If $L$ is non-planar, then the boolean lattice $B_3$ of rank $3$ is a sublattice of $L$.  As was seen in Introduction, ${\mathcal R}_K(B_3)$ cannot be a polynomial ring.  Thus ${\mathcal R}_K(L)$ cannot be a polynomial ring.  Let $L$ be planar.  If the divisor lattice $D_{2\cdot3^3}$ is a sublattice of $L$, then a Pl\"ucker relation appears in the defining ideal $I_{{\mathcal R}_K(L)}$ of ${\mathcal R}_K(L)$, as shown in Example \ref{ASL_again}.  

    Suppose that $L$ is planar and that $D_{2\cdot3^3}$ is not a sublattice of $L$.  If $L$ is thin, then ${\mathcal R}_K(L)$ is a polynomial ring (Lemma \ref{chain}).  If $L$ is not thin and $D_{2\cdot3^3}$ is not a sublattice of $L$, then $L=D_{2^2\cdot 3^2}$ and the desired result is Example \ref{ex: D22}.
\end{proof}

    As mentioned in Introduction, a final goal of our project is to classify those distributive lattices $L$ for which the defining ideal $I_{{\mathcal R}_K(L)}$ of ${\mathcal R}_K(L)$ is generated by quadrics. As first step toward this purpose, we show that the defining ideal $I_{{\mathcal R}_K(L)}$ of a thin distributive lattice $L$ is generated by quadrics.  Recall that a finite simple distributive lattice $L$ is thin if $\theta(L)=2$.

\begin{theorem} \label{highlight}
    Suppose that a finite simple  distributive lattice $L$ is thin.  Then the defining ideal $I_{{\mathcal R}_K(L)}$ of ${\mathcal R}_K(L)$ is generated by quadrics and has a quadratic Gr\"obner basis.  In particular, ${\mathcal R}_K(L)$ is a Koszul algebra. Furthermore, ${\mathcal R}_K(L)$ is Gorenstein.
\end{theorem}

\begin{proof}
    Lemma \ref{weakly} together with Theorem \ref{finally} guarantees that the defining ideal $I_{{\mathcal R}_K(L)}$ of ${\mathcal R}_K(L)$ is generated by quadrics that form a Gr\"obner basis.  In particular, by a result of Fröberg, ${\mathcal R}_K(L)$ is Koszul \cite{Froberg1999Koszul}. Furthermore, since the poset $Q_n$ on which ${\mathcal R}_K(L)$ is ASL is pure, it follows from the Gorenstein criterion that ${\mathcal R}_K(L)$ is Gorenstein \cite[p.~105]{hibi1987}.
\end{proof}

\begin{example} \label{theta(L)=3}
    Let $L$ and $L'$ be the nonplanar distributive lattices of Figure \ref{fig: Nonplanar distributive lattice} and Figure \ref{fig: planar distributive lattice theta =3}, respectively. They both satisfy $\theta(L) = \theta(L') = 3$. However, the defining ideal $I_{{\mathcal R}_K(L)}$ of ${\mathcal R}_K(L)$ is generated by quadrics, but the defining ideal $I_{{\mathcal R}_K(L')}$ of ${\mathcal R}_K(L')$ cannot be generated by quadrics.

\begin{figure}[ht]
    \centering
    \begin{tikzpicture}[scale=0.8]
    \node[draw,shape=circle] (1) at (1,-1) {};
    \node[draw,shape=circle] (2) at (0,0) {};
    \node[draw,shape=circle] (3) at (2,0) {};
    \node[draw,shape=circle] (4) at (1,1) {};
    \node[draw,shape=circle] (5) at (0,1) {};
    \node[draw,shape=circle] (6) at (1,0) {};
    \node[draw,shape=circle] (7) at (2,1) {};
    \node[draw,shape=circle] (8) at (1,2) {};
    \node[draw,shape=circle] (9) at (2,-2){};
    \node[draw,shape=circle] (10) at (3,-1) {};
    \node[draw,shape=circle] (11) at (2,3) {};
    \node[draw,shape=circle] (12) at (3,2) {};
    
    \draw (1)--(2)--(5)--(8)--(7)--(3)--(1);
    \draw (8)--(4)--(2);
    \draw (4)--(3);
    \draw (1)--(6)--(5);
    \draw (6)--(7);
    \draw (1)--(9)--(10)--(3);
    \draw (8)--(11)--(12)--(7);
    \end{tikzpicture}
    \caption{Nonplanar distributive lattice with $\theta(L)=3$} \label{fig: Nonplanar distributive lattice}
\end{figure}
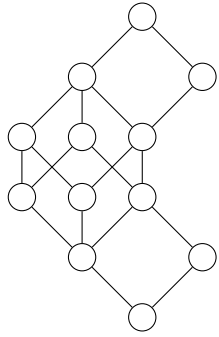

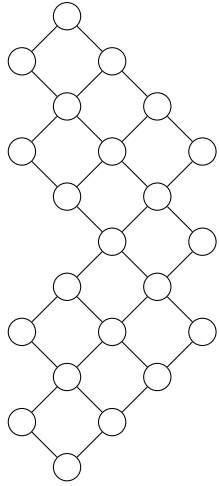
\begin{figure}[ht]
    \centering
    \begin{tikzpicture}[scale=0.6]
    \begin{scope}[rotate around z=0]
    \node[draw,shape=circle] (1) at (0,0) {};
    \node[draw,shape=circle] (2) at (-1,1) {};
    \node[draw,shape=circle] (3) at (1,1) {};
    \node[draw,shape=circle] (4) at (0,2) {};
    \node[draw,shape=circle] (5) at (2,2) {};
    \node[draw,shape=circle] (6) at (-1,3) {};
    \node[draw,shape=circle] (7) at (1,3) {};
    \node[draw,shape=circle] (8) at (3,3) {};
    \node[draw,shape=circle] (9) at (0,4){};
    \node[draw,shape=circle] (10) at (2,4) {};
    \node[draw,shape=circle] (11) at (1,5) {};
    \node[draw,shape=circle] (12) at (3,5) {};
    \node[draw,shape=circle] (13) at (0,6) {};
    \node[draw,shape=circle] (14) at (2,6) {};
    \node[draw,shape=circle] (15) at (-1,7) {};
    \node[draw,shape=circle] (16) at (1,7) {};
    \node[draw,shape=circle] (17) at (3,7) {};
    \node[draw,shape=circle] (18) at (0,8) {};
    \node[draw,shape=circle] (19) at (2,8) {};
    \node[draw,shape=circle] (20) at (-1,9) {};
    \node[draw,shape=circle] (21) at (1,9) {};
    \node[draw,shape=circle] (22) at (0,10) {};
    \end{scope}
    \draw (1)--(3)--(5)--(8)--(10)--(7)--(4)--(2)--(1);
    \draw (3)--(4);
    \draw (5)--(7);
    \draw (4)--(6)--(9)--(11)--(14)--(12)--(10);
    \draw (9)--(7);
    \draw (11)--(10);
    \draw (11)--(13)--(15)--(18)--(20)--(22)--(21)--(19)--(17)--(14);
    \draw (13)--(16)--(19);
    \draw (14)--(16)--(18);
    \draw (18)--(21);
    \end{tikzpicture}
    \caption{Planar distributive lattice with $\theta(L')=3$} \label{fig: planar distributive lattice theta =3}
\end{figure}

\end{example}

\begin{lemma} \label{lemma: theta(L) = r}
    Let $L$ be a finite simple planar distributive lattice and suppose that $\theta(L)  = r $.  Then the divisor lattice $D_{2^{r-1}\cdot 3^{r-1}}$ is an interval of $L$. 
 \end{lemma}

\begin{proof}
    Let $P_L$ be the poset of all join-irreducible elements of $L$.  One has $L={\mathcal J}(P_L)$.  Since $L$ is planar, it follows from Dilworth's theorem  \cite[Chapter 3, Exercise 77]{EC1} that $P_L$ can be decomposed into the disjoint union of chains $C$ and $C'$ of $P_L$.  Let $C:x_0 < \cdots < x_t$ and $C':y_0 < \cdots < y_s$.  If $a \in L$ is equal to $I \in {\mathcal J}(P_L)$ and if $x_p$ (resp. $y_q$) is the maximal element of $I \cap C$ (resp. $I \cap C'$), then we employ the notation $a=(x_p,y_q)$, where $p=-1$ if $I \cap C = \emptyset$.  Now, suppose that $\rho_L(i) = r \geq 2$.  Let $a_1,  \ldots, a_{r}$ denote the elements belonging to $\{ a \in L : {\rm rank}_L(a)= i \}$.  Let $a_j = (x_{p_j},y_{q_j}), 1 \leq j \leq r$, where $p_j + q_j = i - 2$.  Let $p_1  < \cdots < p_r$ and $q_1 > \cdots > q_r$.  Since $a_1 = (x_{p_1},y_{q_1})$ and $a_r = (x_{p_r},y_{q_r})$ are poset ideals of $P_L$, it follows that each $x \in C$ with $x_{p_1} < x \leq x_{p_r}$ and each $y \in C'$ with $y_{q_r} < y \leq y_{q_1}$ are incomparable in $P_L$.  In particular, both $p_1,  \ldots, p_r$ and $q_r, \ldots,  q_1$ are successive integers.  Furthermore, 
\[
    I_*=\{x_0, x_1, \ldots,x_{p_1}, y_0, y_1, \ldots, y_{q_r}\}, \quad I^*= \{x_0, x_1, \ldots,x_{p_r}, y_0, y_1, \ldots, y_{q_1}\} 
\]
    are poset ideals of $P_L$. Let $a_*=(x_{p_1}, y_{q_r}) \in L$ and $a^*=(x_{p_r}, y_{q_1}) \in L$. It then follows that the interval $[a_*,a^*]$ of $L$ is the divisor lattice $D_{2^{r-1}\cdot 3^{r-1}}$.
\end{proof}

\begin{lemma} \label{lemma: theta(L) = 3}
    Let $L$ be a finite simple planar distributive lattice and suppose that there is $i_0$ with $\rho_L(i_0)=\rho_L(i_0+1)=3$.  Then the divisor lattice $D_{2^{2}\cdot 3^{3}}$ is an interval of $L$.
\end{lemma}

\begin{proof}
    We follow the proof of Lemma \ref{lemma: theta(L) = r}.  Let $a = (x_{p-1}, y_{q+1}), b=(x_{p},y_{q}), c=(x_{p+1}, y_{q-1})$ be the elements of $L$ with ${\rm rank}_L(a) = {\rm rank}_L(b) = {\rm rank}_L(c) = i_0$.  Then $a' = (x_{p}, y_{q+1}), c'=(x_{p+1}, y_{q})$ are elements of $L$ with ${\rm rank}_L(a') = {\rm rank}_L(c') = i_0+1$.  Since $\rho_L(i_0+1)=3$, it follows that either $a''=(x_{p+2}, y_{q-1})$ or $c''=(x_{p-1}, y_{q+2})$ must belong $L$.  Say, $a''=(x_{p+2}, y_{q-1}) \in L$.  Hence $x_p < x_{p+1} < x_{p+2}$ and $y_{q} < y_{q+1}$ are the disjoint union of chains.  Thus the interval $[f, g]$ of $L$, where $f=(x_{p-1},y_{q-1})$ and $g=(x_{p+2}, y_{q+1})$, is the divisor lattice $D_{2^{2}\cdot 3^{3}}$. 
\end{proof}

\begin{lemma}
    \label{algebra_retract}
    Let $L$ be a finite distributive lattice and let $L'$ be an interval of $L$. If the defining ideal $I_{{\mathcal R}_K(L)}$ of ${\mathcal R}_K(L)$ is generated by quadrics, then $I_{{\mathcal R}_K(L')}$ is generated by quadrics.   
\end{lemma}

\begin{proof}
    We first recall that, since $L'$ is an interval of $L$, it follows that $\alpha, \beta \in L$ belong to $L'$ if and only if $\alpha\wedge\beta$ and $\alpha\vee\beta$ belong to $L'$.  We claim ${\mathcal R}_K(L') \subset {\mathcal R}_K(L)$ is an algebra retract \cite[p.~747]{algebra_retract}.  Let $S=K[x_i : x_i \in L]$ and $S'=K[x_j : x_j \in L']$ be polynomial rings over $K$.  
    We consider the natural epimorphism $p : S \to S'$, fixing $x_i$ if it belongs to $L'$ and mapping $x_i$ to $0$ otherwise. If $\alpha,\beta$ are in $L'$, then $p(f_{ij})=f_{ij}$, otherwise, one among $x_{\alpha\vee \beta}$ and $x_{\alpha\wedge \beta}$ is not in $L'$ and $p(f_{ij})=0$. Thus the image of ${\mathcal R}_K(L)$ under $p$ is ${\mathcal R}_K(L')$ and the restriction of $p$ to ${\mathcal R}_K(L')$ is the identity.
    It follows that $\varepsilon = p_{|{\mathcal R}_K(L')}$ is a retraction map for ${\mathcal R}_K(L') \subset {\mathcal R}_K(L)$. The desired result follows from comparing the graded Betti numbers of the defining ideals of ${\mathcal R}_K(L')$ and $ {\mathcal R}_K(L)$ \cite[Corollary 2.5]{algebra_retract}.
\end{proof}

\begin{theorem} \label{theta(L)>3}
    Let $L$ be a finite simple planar distributive lattice and suppose that one of the following conditions is satisfied:
 \begin{itemize}
     \item[(i)] $\theta(L) > 3$;
     \item[(ii)] $\theta(L)=3$ and there is $i_0$ with $\rho_L(i_0)=\rho_L(i_0+1)=3$.
 \end{itemize}
    Then the defining ideal $I_{{\mathcal R}_K(L)}$ of ${\mathcal R}_K(L)$ cannot be generated by quadrics.
 \end{theorem}

\begin{proof}
    Lemmas \ref{lemma: theta(L) = 3} and \ref{lemma: theta(L) = r} guarantee that an interval of $L$ is the divisor lattice $D_{2^{2}\cdot 3^{3}}$.

    In \cite{BRUNS2013171} the authors proved that the ideal $I_{\mathcal R_K(D_{2^2 \cdot 3^3})}$ has minimally cubic relations, which, together with the Plücker quadrics, form a system of generators 
\cite{2x2minors}.

    Then it follows from Lemma \ref{algebra_retract} that $I_{{\mathcal R}_K(L)}$ cannot be generated by quadrics. 
\end{proof}

    We close the present paper with the following conjecture. 

\begin{conjecture}
    Let $L$ be a simple planar distributive lattice.  Then the defining ideal $I_{{\mathcal R}_K(L)}$ of ${\mathcal R}_K(L)$ is generated by quadrics if and only if at most one interval is $D_{2^{2}\cdot 3^{2}}$.
\end{conjecture}

\section*{Appendix}
    We give a short proof of the classical result that the homogeneous coordinate ring of the Grassmannian ${\rm Gr}_K(d,n)$ is an ASL on a distributive lattice over $K$.

    Let $2 \leq d \leq n$ be integers.  Let $L$ denote the set of those symbols $[i_1\cdots i_d]$ with $1\leq i_1 < \cdots < i_d \leq n$.  We introduce the partial order on $L$ by setting $[i_1\cdots i_d] \leq [i'_1\cdots i'_d]$ if $i_1 \leq i'_1, \ldots, i_d \leq i'_d$.  It then follows that $L$ is a distributive lattice and $[i_1\cdots i_d] \in L$ is join-irreducible if and only if $[i_1 \cdots i_d]$ is one of the following:
\begin{itemize}
    \item $[a,a+1,\ldots, a + (d-1)], 1 <  a \leq n - d + 1$;  
    \item $[1, 2, \ldots, a, a + b, a + b + 1, \cdots, b + d - 1], 1 < a < d, 1 < b \leq n - d + 1$. 
\end{itemize}

    Let $X=\{x_{ij}\}_{1 \leq i \leq d, 1 \leq j \leq n}$ denote the $d \times n$ matrix of variables and $[i_1,\ldots,i_d]$, where $1\leq i_1 < \cdots < i_d \leq n$, the determinant of the $d \times d$ submatrix of $X$ consisting of columns $i_1, \ldots, i_d$.  Let $S= K[x_{ij} : 1 \leq i \leq d, 1 \leq j \leq n]$ denote the polynomial ring in $dn$ variables over a field $K$.  Let $A$ denote the subring of $S$ generated by those determinants $[i_1,\ldots,i_d]$ with $1\leq i_1 < \cdots < i_d \leq n$.  Now, considering $L$ to be a subset of $A$ in an obvious way. Pl\"ucker relations guarantee that $A$ is a weakly ASL on $L$ over $K$. Since there is a monomial order $<$ on $S$ for which ${\rm in}_{<}([i_1,\ldots,i_d]) = x_{1,i_1} x_{2,i_2} \cdots x_{d,i_d}$ \cite[p.~80]{bruns2022determinants}, by virtue of Lemma \ref{weakly}, in order to show that $A$ is an ASL on $L$ over $K$, it suffices to prove the following Lemma.

\begin{lemma} \label{toric}
    Define the injection $\psi:L \to S$ by setting $\psi([i_1,\ldots, i_d]) = \prod_{j=1}^{d} x_{j,i_j}$. Then the toric ring $K[\psi(\xi) : \xi \in L]$ is an ASL on $L$ over $K$. 
\end{lemma}

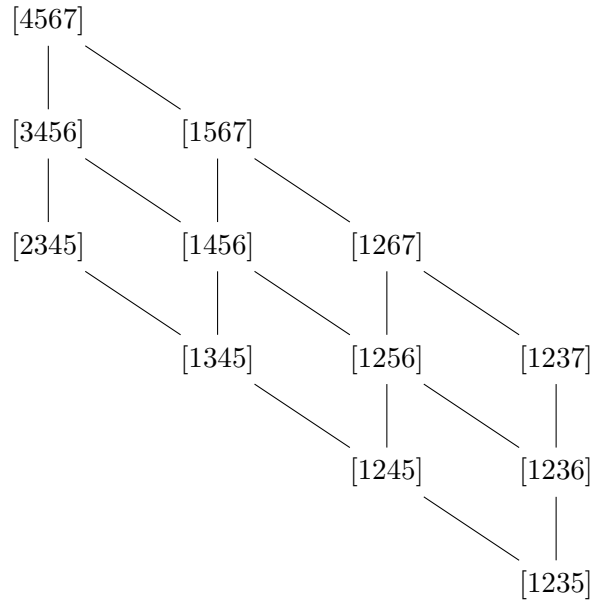
\begin{figure}
\centering
\begin{tikzpicture}[scale=1.5]
\node (x0) at (0,0) {$[4567]$};
\node (x1) at (0,-1) {$[3456]$};
\node (x2) at (0,-2) {$[2345]$};
\node (x3) at (1.5,-1) {$[1567]$};
\node (x4) at (1.5,-2) {$[1456]$};
\node (x5) at (1.5,-3) {$[1345]$};
\node (x6) at (3,-2) {$[1267]$};
\node (x7) at (3,-3) {$[1256]$};
\node (x8) at (3,-4) {$[1245]$};
\node (x9) at (4.5,-3) {$[1237]$};
\node (x10) at (4.5,-4) {$[1236]$};
\node (x11) at (4.5,-5) {$[1235]$};
\draw(x0)--(x1)--(x2)--(x5)--(x8)--(x11)--(x10)--(x9)--(x6)--(x3)--(x0);
\draw(x3)--(x4)--(x5);
\draw(x6)--(x7)--(x8);
\draw(x1)--(x4)--(x7)--(x10);
\end{tikzpicture}
\caption{The poset $P$ with $d=4$ and $n=7$}
\label{d=4_n=7.tex}
\end{figure}

\begin{proof}
Let $P$ denote the set of join-irreducible elements of $L$.  Then $P$ can be decomposed into the disjoint union of chains $P_1, P_2, \ldots, P_{d}$, where 
\begin{eqnarray*}
 P_1 &=& \{[2,3,\ldots,d+1],[3,4,\ldots, d+2],\ldots,[n-d+1, n-d+2, \ldots, n]\}, \\
 P_2&=& \{[1,3,4,\ldots,d+1], [1,4,5,\ldots,d+2], \ldots, [1, n-d+2,n-d+3, \ldots, n ]\},\\
 P_3&=& \{[1,2,4,5,\ldots,d+2], [1,2,5,6,\ldots,d+3], \ldots, [1,2, n-d+3, \ldots, n]\},\\
 & \vdots & \\
 P_{d}&=&\{[1,2,3,\ldots,d-1, d+1], [1,2,3,\ldots,d-1, d+2], \ldots, [1,2,3,\ldots,d-1, n]\}.
\end{eqnarray*}
We introduce the symbols $x_{ij}, 1\leq i \leq d, 1 \leq j \leq n$ and express each $P_i$ in the form   
\[ 
P_i : x_{i, i+1} < x_{i,i+2} < \cdots < x_{i,n-d+i}. 
\]
    For example, in Figure \ref{d=4_n=7.tex}, one has $[1345] = x_{2,3}$ and $[1256] = x_{3,5}$.  Furthermore, we add a unique minimal element $\hat{0} = x_{1,1}x_{2,2} \cdots x_{d,d}$ to $P$.  
    We define the injection $\varphi:L \to S$ by setting $\varphi([i_1,\ldots, i_d]) = f_1f_2 \cdots f_d,$ where
\begin{eqnarray*}
    f_j = x_{j,j} x_{j,j+1} x_{j,j+2} \cdots x_{j,i_j}.
\end{eqnarray*}
    Since $x_{j,j+q} = [1,2,\ldots, j-1, j+q, j+q+1, \ldots]\leq [i_1,\ldots, i_d]$ if and only if $j+q \leq i_j$, it follows that the toric ring $K[\varphi(p) : p \in L]$ coincides with $\mathcal R_K(L)$ introduced in \cite{hibi1987}.  In particular, $K[\varphi(\xi)] : \xi \in L]$ is an ASL on $L$ over $K$.  On the other hand, defining the injection $\psi:L \to S$ by setting $\psi([i_1,\ldots, i_d]) = \prod_{j=1}^{d} x_{j,i_j}$, one clearly has $K[\varphi(\xi) : \xi \in L] = K[\psi(\xi) : \xi \in L]$.
\end{proof}

\printbibliography

\medskip
\noindent{\bf Authors' addresses:}
\medskip

\end{document}